\documentclass[11pt,english]{amsart}
\usepackage[T1]{fontenc}
\usepackage[latin9]{inputenc}
\usepackage{amsthm}
\usepackage{amstext}
\usepackage{esint}

\makeatletter
\usepackage{amsmath,amsfonts,amssymb,amsthm,epsfig}

\usepackage{color}
\usepackage[final,allcolors=blue,colorlinks=true]{hyperref}

\def\R{\mathbb R}
\def\N{\mathbb N}
\def\C{\mathbb C}

\def\S{\mathbb S}
\def\X{\mathbb X}

\def\al{\alpha}

\def\de{\delta}

\def\la{\lambda}
\def\si{\sigma}

\def\om{\omega}
\def\na{\nabla}
\def\Om{\Omega}  
\def\De{\Delta}      

\def\cal{\mathcal}
\def\L{\mathcal L}                                       
\def\E{\mathcal E}                                      
\def\X{\mathbb X}                                       
\def\wq{\infty}
\def\pa{\partial}

\def\loc{\text{\rm loc}}
\def\rad{\text{\rm rad}}
\newcommand{\Ker}{\text{\rm Ker}}

\newcommand{\D}{{\rm d}}

\newcommand{\medint}{-\kern -,375cm\int}         
\newcommand{\medintinrigo}{-\kern -,315cm\int}
\newcommand{\wto}{\rightharpoonup}                
\newcommand\esssup{\text{\rm \,esssup\,}}

\numberwithin{equation}{section}
\newtheorem{theorem}{Theorem}[section]

\newtheorem{definition}[theorem]{Definition}

\newtheorem{lemma}[theorem]{Lemma}

\newtheorem{proposition}[theorem]{Proposition}

\theoremstyle{definition}

\makeatother

\usepackage{babel}
\begin{document}
\title[Quasilinear Schr\"odinger Equations]{Remarks on Nondegeneracy of Ground States for Quasilinear Schr\"odinger Equations} 

           \author[C.-L. Xiang]{Chang-Lin Xiang}      

\address[]{University of Jyv\"askyl\"a, Department of Mathematics and Statistics, P.O. Box 35, FI-40014 University of Jyv\"askyl\"a, Finland.}
\email[]{Xiang\_math@126.com}

\begin{abstract}
In this paper, we answer affirmatively the problem proposed by A. Selvitella in his paper "Nondegenracy of the ground state for quasilinear Schr\"odinger Equations" (see Calc. Var. Partial Differ. Equ., {\bf 53} (2015),   pp 349-364): every ground state  of  equation \begin{eqnarray*}-\Delta u-u\Delta |u|^2+\omega u-|u|^{p-1}u=0&&\text{in }\mathbb{R}^N\end{eqnarray*} is nondegenerate for $1<p<3$, where  $\omega>0$ is a given constant and $N\ge1$. We also derive further properties on the linear operator associated to ground states  of above equation.
\end{abstract}

\maketitle

{\small    
\keywords {\noindent {\bf Keywords:} Quasilinear Schr\"odinger  equations; Ground states;  Nondegeneracy}
\smallskip
\newline
\subjclass{\noindent {\bf 2010 Mathematics Subject Classification:}  Primary 35J62, 35Q55; Secondary 35J60}
\tableofcontents}
\bigskip

\section{Introduction and main result}

Consider the quasilinear Schr\"odinger equation
\begin{eqnarray}
i\pa_{t}U=-\De U-U\De|U|^{2}-|U|^{p-1}U &  & \text{in }\R^{N}\times\R_{+},\label{eq: Quasili-Schrodinger eq.}
\end{eqnarray}
where $U:\R^{N}\times\R_{+}\to\C$ is the wave function, $i$ is the
imaginary unit, $N\ge1$ and $p>1$. Equation (\ref{eq: Quasili-Schrodinger eq.})
arises in various domains of physics, such as superfluid film equation
in plasma physics. More physical background of equation (\ref{eq: Quasili-Schrodinger eq.})
can be found in e.g. Colin et al. \cite{Colin-Jeanjean- Squassina-2010}
and the references therein. Equation (\ref{eq: Quasili-Schrodinger eq.})
has been studied extensively in the literature, see e.g. \cite{Colin-2002,Colin-Jeanjean- Squassina-2010,Kenig-Ponce-Vega-2004,Lange-Poppenberg-Teismann-1999,Poppenberg-2001,Poppenberg-Schmitt- Wang-2002}
and the references therein. A special class of solutions to equation
(\ref{eq: Quasili-Schrodinger eq.}) that represent particles at rest
is the so called standing waves, that is, solutions of the form
\[
U(x,t)=e^{i\om t}u(x),
\]
where $\om>0$ is a given constant which stands for the time frequency,
and $u:\R^{N}\to\C$ is a complex valued function that is independent
of time $t\in\R_{+}$. It is elementary to verify that if $U(x,t)=e^{i\om t}u(x)$
is a standing wave, then $u$ solves the stationary equation 
\begin{eqnarray}
-\Delta u-u\Delta|u|^{2}+\omega u-|u|^{p-1}u=0 &  & \text{in }\R^{N}.\label{eq: Object equ.}
\end{eqnarray}

Equation (\ref{eq: Object equ.}) is known \cite{Colin-Jeanjean- Squassina-2010}
as the Euler-Lagrange equation of the energy functional $\E_{\om}:\X_{\C}\to\R$
defined as
\[
{\cal E}_{\om}(u)=\frac{1}{2}\int_{\R^{N}}|\na u|^{2}\D x+\int_{\R^{N}}|u|^{2}|\na|u||^{2}\D x+\frac{\om}{2}\int_{\R^{N}}|u|^{2}\D x-\frac{1}{p+1}\int_{\R^{N}}|u|^{p+1}\D x,
\]
where $\X_{\C}$ is the function space given by 
\[
\X_{\C}=\left\{ u\in H^{1}(\R^{N}):\int_{\R^{N}}|u|^{2}|\na|u||^{2}\D x<\wq\right\} .
\]
Here we assume throughout the paper that $1<p<p_{\max}$ holds, where
the critical exponent $p_{\max}$ is defined as 
\begin{equation}
p_{\max}\equiv\begin{cases}
\frac{3N+2}{N-2} & \text{if }N\ge3\\
\wq & \text{if }N=1,2.
\end{cases}\label{eq: p-max}
\end{equation}
Technically speaking, the condition that $p$ be strictly less than
$p_{\max}$ ensures that the power nonlinearity in equation (\ref{eq: Object equ.})
is $\X_{\C}$-subcritical. Indeed, by a simple application of Sobolev
embedding theorems, we infer that $\X_{\C}$ is continuously embedded
into $L^{p+1}(\R^{N})$ for $1<p<\wq$ if $N=1,2$ and $1<p\le(3N+2)/(N-2)$
if $N\ge3$. 

In view of the variational structures of equation (\ref{eq: Object equ.}),
critical point theory has been devoted to find solutions for equation
(\ref{eq: Object equ.}). Here, as in Colin et al. \cite{Colin-Jeanjean- Squassina-2010},
a function $u\in\X_{\C}$ is said to be a solution to equation (\ref{eq: Object equ.}),
if for any function $\phi\in C_{0}^{\wq}(\R^{N})$, the space of smooth
functions in $\R^{N}$ with compact support, there holds
\[
\text{{\rm Re}}\int_{\R^{N}}\Big(\na u\cdot\na\bar{\phi}+\na|u|^{2}\cdot\na(u\bar{\phi})+\om u\bar{\phi}-|u|^{p-1}u\bar{\phi}\Big)\D x=0
\]
(here $\text{{\rm Re}}z$ is the real part of $z\in\C$). The existence
of solutions to equation (\ref{eq: Object equ.}) is now well known,
see e.g. \cite{Colin-Jeanjean-2004,Colin-Jeanjean- Squassina-2010,Liu-Wang-Wang-2003,Liu-Wang-Wang-2004,Liu- Wang-2003}
and the references therein.

In this paper, we consider ground state to equation (\ref{eq: Object equ.}).
Following the convention of Colin et al. \cite{Colin-Jeanjean- Squassina-2010}
(see also Selvitella \cite{Selvitella-2011,Selvitella-2015}), we
say that a solution $u\in\X_{\C}$ to equation (\ref{eq: Object equ.})
is a \emph{ground state}, if $u$ satisfies 
\[
\E_{\om}(u)=\inf\left\{ \E_{\om}(v):v\in\X_{\C}\text{ is a nontrivial solution to equation (\ref{eq: Object equ.})}\right\} .
\]
We remark that the notion of ground state here is different from that
defined in \cite{Chang et al-2007,Frank-Lenzmann-2013,Frank-Lenzmann-Silvestre-2013,Kwong1989,Lenzmann-2009}.
We are concerned about the nondegeneracy (see below) of ground states.
Before proceeding further, let us summarize the existence result of
ground states to equation (\ref{eq: Object equ.}) together with a
list of basic properties for later use.

\begin{theorem} \label{thm: properties of GS} Assume that $1<p<p_{\max}$
with $p_{\max}$ defined as in (\ref{eq: p-max}). Then for any given
constant $\om>0$, there exists a ground state to equation (\ref{eq: Object equ.}).
Moreover, for any ground state $u\in\X_{\C}$ to equation (\ref{eq: Object equ.}),
there exist a constant $\theta\in\R$, a decreasing positive function
$v:[0,\wq)\to(0,\wq)$ and a point $x_{0}\in\R^{N}$ such that $u$
is of the form 
\begin{eqnarray*}
u(x)=e^{i\theta}v(|x-x_{0}|) &  & \text{for }x\in\R^{N}.
\end{eqnarray*}
Furthermore, the following properties hold.

(1) (Smoothness) $u\in C^{\wq}(\R^{N})$. 

(2) (Decay) For any multi-index $\al\in\N^{N}$ with $|\al|\ge0$,
there exist positive constants $C_{\al}>0$ and $\de_{\al}>0$ such
that 
\begin{eqnarray*}
|\pa^{\al}u(x)|\le C_{\al}\exp(-\de_{\al}|x|) &  & \text{for all }x\in\R^{N}.
\end{eqnarray*}

(3) (Uniqueness) In the case $N=1$, the ground states to equation
(\ref{eq: Object equ.}) is unique up to phase and translation. In
particular, there exists a unique positive even ground state for equation
(\ref{eq: Object equ.}).\end{theorem}

For a complete proof of Theorem \ref{thm: properties of GS}, we refer
to Colin et al. \cite{Colin-Jeanjean- Squassina-2010} and Selvitella
\cite{Selvitella-2011}.

In this paper, our aim is to study nondegeneracy of ground states
for equation (\ref{eq: Object equ.}). The motivation comes from the
fact that the nondegeneracy of ground states for equation (\ref{eq: Object equ.})
plays an important role when studying the existence of concentrating
solutions in the semiclassical regime. We refer the readers to Selvitella
\cite{Selvitella-2015} for more applications of nondegeneracy results.
We also follow the convention of Selvitella \cite{Selvitella-2015}
(see also Ambrosetti and Malchiodi \cite{Ambrosetti-Malchiodi-book})
and define nondegeneracy of ground states for equation (\ref{eq: Object equ.})
as follows.

\begin{definition} \label{def: nondegeneracy}Let $u\in\X_{\C}$
be a ground state of equation (\ref{eq: Object equ.}). We say that
$u$ is nondegenerate if the following properties hold:

(1) (ND) $\Ker\E_{\om}^{\prime\prime}(u)=\text{{\rm span}}\left\{ iu,\pa_{x_{1}}u,\cdots,\pa_{x_{N}}u\right\} $;

(2) (Fr) $\E_{\om}^{\prime\prime}(u)$ is an index 0 Fredholm map.\end{definition}

The first result on nondegeneracy of ground states for equation (\ref{eq: Object equ.})
was obtained by Selvitella \cite{Selvitella-2011} in a perturbative
setting, where uniqueness of ground states for equation (\ref{eq: Object equ.})
was also considered. In his quite recent paper \cite{Selvitella-2015},
Selvitella proved, under the assumption 
\[
p\ge3,
\]
 that every ground state of equation (\ref{eq: Object equ.}) is nondegenerate
in the sense of Definition \ref{def: nondegeneracy} above, see Theorem
1.2 of \cite{Selvitella-2015}. Selvitella also commented (see Remark
1.3 of \cite{Selvitella-2015}) that his nondegeneracy result could
also be true for the case $1<p<3$. However, his approach can not
handle this case. In this paper, we give an affirmative answer to
his question. We obtain the following result. 

\begin{theorem} \label{thm: main result} For $1<p<3$, every ground
state of equation (\ref{eq: Object equ.}) is nondegenerate in the
sense of Definition \ref{def: nondegeneracy} above.\end{theorem} 

We remark that our argument is applicable to the whole range $1<p<p_{\max}$. 

As already remarked by Selvitella (see Remark 1.3 of \cite{Selvitella-2015}),
except Proposition 3.10 of \cite{Selvitella-2015} that requires him
to assume $p\ge3$, all the rest of his arguments can be applied to
the whole range $1<p<p_{\max}$ to prove Theorem \ref{thm: main result}.
So in this paper, we will follow the line of Selvitella \cite{Selvitella-2015}
to prove Theorem \ref{thm: main result}. However, since his approach
can not handle the whole range $1<p<p_{\max}$, we will apply a different
idea from that of Selvitella \cite{Selvitella-2015}. Precisely, let
$u$ be a positive radial ground state of equation (\ref{eq: Object equ.}).
Define the linear operator $\L_{+}$ associated to $u$ by 
\[
\L_{+}\eta=-\De\eta-2u\De(u\eta)+\om\eta-(\De u^{2}+pu^{p-1})\eta.
\]
We note that $\L_{+}$ is a self-adjoint operator acting on $L^{2}(\R^{N})$
with form domain $\X_{\C}$ and operator domain $H^{2}(\R^{N})$.
It turns out that the key to prove Theorem \ref{thm: main result}
is to show that $\L_{+}$ satisfies 
\begin{equation}
\Ker\L_{+}=\text{{\rm span}}\left\{ \pa_{x_{1}}u,\cdots,\pa_{x_{N}}u\right\} .\label{eq: Kernel 1.1}
\end{equation}
In the approach of Selvitella \cite{Selvitella-2015} to (\ref{eq: Kernel 1.1}),
ordinary differential equation analysis plays a central role, in which
the assumption $p\ge3$ is required. To prove (\ref{eq: Kernel 1.1})
for $p$ in the whole range $1<p<p_{\max}$, we will use a spectrum
analysis to the operator $\L_{+}$. In this way, we obtain deeper
results on the operator $\L_{+}$ than (\ref{eq: Kernel 1.1}). 

Our idea comes from the spectrum analysis of Chang et al. \cite{Chang et al-2007},
even through we can not use their refined arguments to derive (\ref{eq: Kernel 1.1})
directly. Roughly speaking, Chang et al. \cite{Chang et al-2007}
considered the following problem. Let $Q$ the unique positive radial
solution in $H^{1}(\R^{N})$ to the equation 
\begin{eqnarray}
-\De Q+\om Q-Q^{q}=0 &  & \text{in }\R^{N},\label{eq: NLS}
\end{eqnarray}
where $1<q<(N+2)/(N-2)$. Here we assume $N\ge3$ for simplicity.
Chang et al. \cite{Chang et al-2007} studied the spectrum of the
linear operator ${\cal A}_{+}$ around $Q$ given by 
\begin{equation}
{\cal A}_{+}\eta=-\De\eta+\om\eta-qQ^{q-1}\eta,\label{eq: operator of Chang}
\end{equation}
acting on $L^{2}(\R^{N})$ with form domain $H^{1}(\R^{N})$ and operator
domain $H^{2}(\R^{N})$. For the importance of the spectrum of ${\cal A}_{+}$,
we refer the readers to Chang et al. \cite{Chang et al-2007}. We
remark that Chang et al. \cite{Chang et al-2007} studied far more
than ${\cal A}_{+}$ in their work. We will give a brief comparison
between the two self-adjoint operators ${\cal A}_{+}$ and $\L_{+}$
below. We also refer the readers to \cite{Frank-Lenzmann-2013,Frank-Lenzmann-Silvestre-2013,Lenzmann-2009}
for spectrum analysis for linearized operators around ground states
of nonlocal problems. 

It seems that the spectrum $\si(\L_{+})$ of $\L_{+}$ has not been
studied in the literature. According to the analysis in next section,
we obtain the following properties for the spectrum of $\L_{+}$:

(1) the continuous spectrum of $\L_{+}$ is contained in $[\om,\wq)$
(see Lemma \ref{lem: essential spectrum});

(2) $\inf\si(\L_{+})<0$ is the first eigenvalue of $\L_{+}$ and
is simple (see Lemma \ref{lem: first eigenvalue is simple and negative}); 

(3) as a consequence of (1) and (2), $0$ belongs to the discrete
spectrum of $\L_{+}$ and is not the first eigenvalue of $\L_{+}$. 

We will give the proof of Theorem \ref{thm: main result} together
with above properties of $\si(\L_{+})$ in the next section. Before
we close this section, let us address some differences between the
two self-adjoint operators ${\cal A}_{+}$ defined as in (\ref{eq: operator of Chang})
and $\L_{+}$. 

First we point out that the second property (2) is not obvious. In
fact, even the fact $\inf\si(\L_{+})<0$ is not obvious. For the operator
${\cal A}_{+}$, a simple observation gives that 
\[
\langle{\cal A}_{+}Q,Q\rangle=-(q-1)\int_{\R^{N}}Q^{q+1}\D x<0,
\]
which implies $\inf\si(A_{+})<0$. Furthermore, it is standard (see
e.g. Lieb and Loss \cite{Lieb-Loss-book}) to show that $\inf\si({\cal A}_{+})$
is the first eigenvalue of ${\cal A}_{+}$ and is simple. However,
in our case, a direct calculation gives us 
\[
\langle\L_{+}u,u\rangle=8\int_{\R^{N}}u^{2}|\na u|^{2}\D x-(p-1)\int_{\R^{N}}u^{p+1}\D x.
\]
We will confirm that $\langle\L_{+}u,u\rangle<0$ in the case $N=2$,
based on a Pohozaev type identity, see the end of next section. While
in higher dimensions $N\ge3$, it is still not clear from above expression
whether $\langle\L_{+}u,u\rangle$ is negative or not. Hence we can
not infer that $\inf\si(\L_{+})<0$ holds by such a simple observation
as above. Similarly, due to the presence of the quasilinear term $-2u\De(u\eta)$
in $\L_{+}$, it is not obvious as well that a nonnegative eigenfunction
of $\L_{+}$ is in fact positive everywhere. 

Second, we point out that we do not known whether $0$ is the second
eigenvalue of $\L_{+}$ or not. Then we can not give exact estimates
on the numbers of nodal domains of radial functions $v$ with $v\in\Ker\L_{+}$.
Thus we can not use the arguments of Chang et al. \cite{Chang et al-2007}
directly (see the proof of Lemma 2.1 of Chang et al. \cite{Chang et al-2007}).
As to the operator ${\cal A}_{+}$, it is known (see e.g. Chang et
al. \cite{Chang et al-2007}) that $0$ is the second eigenvalue of
${\cal A}_{+}$. This is due to fact that, by uniqueness, $Q$ is
also a minimizer (up to rescaling) of the 'Weinstein' functional 
\begin{eqnarray*}
W(f)=\frac{\|\na f\|_{2}^{(q-1)N/2}\|f\|_{2}^{(N+2-(N-2)q)/2}}{\|f\|_{q+1}^{q+1}}, &  & f\in H^{1}(\R^{N}),f\not\equiv0.
\end{eqnarray*}
In our case, except the case $N=1$ (see Colin et al. \cite{Colin-Jeanjean- Squassina-2010}),
the uniqueness of positive solutions to equation (\ref{eq: Object equ.})
is unknown in general when $N\ge2$. In fact, even the uniqueness
of ground states to equation (\ref{eq: Object equ.}) is unknown in
general when $N\ge2$. Some partial results on uniqueness were obtained
in the literature. Since this is out of the scope of this paper, we
refer the interested readers to Selvitella \cite{Selvitella-2011,Selvitella-2015}
and the references therein. 

Our nations are standard. We write $\R_{+}=(0,\wq)$ and denote $\N=\{0,1,2,\cdots\}$
the set of nonnegative integers. For any $1\le s\le\infty$, $L^{s}(\R^{N})$
is the Banach space of complex valued Lebesgue measurable functions
$u$ such that the norm 
\[
\|u\|_{s}=\begin{cases}
\left(\int_{\R^{N}}|u|^{s}\D x\right)^{\frac{1}{s}} & \text{if }1\le s<\infty\\
\esssup_{\R^{N}}|u| & \text{if }s=\infty
\end{cases}
\]
is finite. A function $u$ belongs to the Sobolev space $H^{k}(\R^{N})$
($k\in\N$) if $u\in L^{2}(\R^{N})$ and its weak partial derivatives
up to order $k$ also belong to $L^{2}(\R^{N})$. We equip $H^{k}(\R^{N})$
with the norm 
\[
\|u\|_{H^{k}}=\sum_{\al\in\N^{N},|\al|\le k}\|\pa^{\al}u\|_{2}.
\]
For the properties of the Sobolev functions, we refer to the monograph
\cite{Ziemer}. By abuse of notation, we write $f(x)=f(r)$ with $r=|x|$
whenever $f$ is a radially symmetric function in $\R^{N}$.

\section{Proof of main result}

In this section we prove Theorem \ref{thm: main result}. Since we
deal with the same problem as that of Selvitella \cite{Selvitella-2015},
we will follow the line of Selvitella \cite{Selvitella-2015}. Similar
lines can also be found in e.g. \cite{Chang et al-2007,Frank-Lenzmann-2013,Frank-Lenzmann-Silvestre-2013,Lenzmann-2009}
and the monograph \cite{Ambrosetti-Malchiodi-book}.

\subsection{Proof of Theorem \ref{thm: main result}}

Let $u\in\X_{\C}$ be an arbitrary ground state for equation (\ref{eq: Object equ.}).
By Definition \ref{def: nondegeneracy}, to prove Theorem \ref{thm: main result}
we have to show that $\E_{\om}^{\prime\prime}(u)$ satisfies property
(ND) and property (Fr). The property (Fr) can be proved by the same
argument as that of Selvitella \cite{Selvitella-2015}, since which
is applicable to the whole range $1<p<p_{\max}$. So we omit the details. 

We focus on the proof of the property (ND), that is, we prove in the
following that 
\begin{equation}
\Ker\E_{\om}^{\prime\prime}(u)=\text{{\rm span}}\left\{ iu,\pa_{x_{1}}u,\cdots,\pa_{x_{N}}u\right\} .\label{eq: ND}
\end{equation}

By Theorem \ref{thm: properties of GS}, every ground state of equation
(\ref{eq: Object equ.}) can be regarded as a positive, radial and
symmetric-decreasing ground state. Hence we assume in the sequel that
$u=u(|x|)>0$ is a positive, radial and symmetric-decreasing ground
state for equation (\ref{eq: Object equ.}). We also assume $N\ge2$
in the sequel. In the case $N=1$ the proof of (\ref{eq: ND}) is
similar and even simpler. Then the linearized operator $\E_{\om}^{\prime\prime}(u)$
is giving by 
\[
\E_{\om}^{\prime\prime}(u)\xi=-\De\xi-2u\De\left(u\text{Re}\xi\right)+\om\xi-\left(\De u^{2}\right)\xi-(p-1)u^{p-1}\text{Re}\xi-u^{p-1}\xi
\]
acting on $L^{2}(\R^{N})$ with form domain $\X_{\C}$ and operator
domain $H^{2}(\R^{N})$. 

Note that $\E_{\om}^{\prime\prime}(u)$ is not even $\C$-linear.
To overcome this difficulty, it is preferable to introduce the linear
operator $\L_{+}$ given by 
\begin{equation}
\L_{+}\eta=-\De\eta-2u\De(u\eta)+\om\eta-(\De u^{2}+pu^{p-1})\eta,\label{operator: real part of linearization}
\end{equation}
acting on $L^{2}(\R^{N})$ with form domain $\X_{\C}$ and operator
domain $H^{2}(\R^{N})$, and the linear operator $\L_{-}$ given by
\[
\L_{-}\zeta=-\De\zeta+\om\zeta-(\De u^{2}+u^{p-1})\zeta
\]
acting on $L^{2}(\R^{N})$ with form domain $H^{1}(\R^{N})$ and operator
domain $H^{2}(\R^{N})$. Then for any $\xi\in H^{2}(\R^{N})$ we obtain
\[
\E_{\om}^{\prime\prime}(u)\xi=\L_{+}\text{Re}\xi+i\L_{-}\text{Im}\xi
\]
(here $\text{Im}z$ is the imaginary part of $z\in\C$). Therefore,
to prove (\ref{eq: ND}), it is sufficient to prove the following
result. 

\begin{proposition} \label{prop: two kernels} Let $\L_{+}$ and
$\L_{-}$ be defined as above. We have that 
\begin{equation}
\Ker\L_{+}=\text{{\rm span}}\left\{ \pa_{x_{1}}u,\cdots,\pa_{x_{N}}u\right\} \label{eq: kernel 1}
\end{equation}
 and 
\begin{equation}
\Ker\L_{-}=\text{{\rm span}}\left\{ u\right\} .\label{eq: kernel 2}
\end{equation}
\end{proposition}

\begin{proof} First we prove (\ref{eq: kernel 2}). The proof is
standard. In fact, we can use the argument of Selvitella \cite{Selvitella-2015}
since which is applicable to $p$ in the whole range of $1<p<p_{\max}$.
We give a proof here for the reader's convenience. 

First we use spherical harmonics to decompose functions $v\in H^{j}(\R^{N})$
for $j\in\N$. Denote by $-\De_{\S^{N-1}}$ the Laplacian-Beltrami
operator on the unit $N-1$ dimensional sphere $\S^{N-1}$ in $\R^{N}$.
Write 
\begin{eqnarray*}
M_{k}=\frac{(N+k-1)!}{(N-1)!k!}\quad\forall\, k\ge0, & \text{and } & M_{k}=0\quad\forall\, k<0.
\end{eqnarray*}
Denote by $Y_{k,l}$, $k=0,1,\ldots$ and $1\le l\le M_{k}-M_{k-2}$,
the spherical harmonics such that 
\[
-\De_{\S^{N-1}}Y_{k,l}=\la_{k}Y_{k,l}
\]
for all $k=0,1,\ldots$ and $1\le l\le M_{k}-M_{k-2}$, where 
\begin{eqnarray*}
\la_{k}=k(N+k-2) &  & \forall\, k\ge0
\end{eqnarray*}
are eigenvalues of $-\De_{\S^{N-1}}$ with multiplicities $M_{k}-M_{k-2}$.
In particular, we deduce that $\la_{0}=0$ is of multiplicity 1 with
$Y_{0,1}=1$, and $\la_{1}=N-1$ is of multiplicity $N$ with $Y_{1,l}=x_{l}/|x|$
for $1\le l\le N$. 

Then for any function $v\in H^{j}(\R^{N})$, we have 
\[
v(x)=v(r\Om)=\sum_{k=0}^{\wq}v_{k}(r)Y_{k}(\Om)
\]
with $r=|x|$ and $\Om=x/|x|$, where 
\begin{eqnarray}
v_{k}(r)=\int_{\S^{N-1}}v(r\Om)Y_{k}(\Om)\D\Om &  & \forall\, k\ge0.\label{eq: k-th component of v}
\end{eqnarray}
Note that $v_{k}\in H^{j}(\R_{+},r^{N-1}\D r)$ holds for all $k\ge0$
since $v\in H^{j}(\R^{N})$. 

Next, apply above decomposition to any function $v\in H^{1}(\R^{N})$.
We conclude that $\L_{-}v=0$ if and only if 
\[
\L_{-,k}v_{k}\equiv-v_{k}^{\prime\prime}-\frac{N-1}{r}v_{k}^{\prime}+\frac{\la_{k}}{r^{2}}v_{k}+\om v_{k}-(\De u^{2}+u^{p-1})v_{k}=0
\]
for all $k\ge0$, where $v_{k}$ is defined as in (\ref{eq: k-th component of v}).
Note that $\L_{-,k}$ is a self-adjoint operator acting on $L^{2}(\R_{+},r^{N-1}\D r)$
for all $k\in\N$.

First we consider $k=0$. In this case we have $\la_{0}=0$. By a
direct computation, we obtain that $\L_{-,0}u=0$. Since $u(r)>0$
for all $r>0$, we conclude in a standard way that $u$ is the first
eigenfunction and $0$ is the first simple eigenvalue of $\L_{-,0}$.
Thus we have 
\begin{equation}
\Ker\L_{-,0}=\text{span}\{u\}.\label{eq: kernel of L---0}
\end{equation}

Next consider $k\ge1$. We claim that for all $k\ge1$, there holds
\begin{equation}
\Ker\L_{-,k}=\{0\}.\label{eq: kernel of L---k}
\end{equation}
Indeed, since $\la_{k}>0$, we deduce that $\L_{-,k}>\L_{-,0}$ holds
in the sense of quadratic form, which implies that $\L_{-,k}w=0$
if and only if $w\equiv0$. This proves the claim. 

Finally, we infer from (\ref{eq: kernel of L---0}) and (\ref{eq: kernel of L---k})
that (\ref{eq: kernel 2}) holds. This finishes the proof of (\ref{eq: kernel 2}).

It remains to prove (\ref{eq: kernel 1}). We still use spherical
harmonic as above. Then $\L_{+}v=0$ for $v\in\X_{\C}(\R^{N})$ if
and only if for all $k=0,1,\ldots$, we have
\begin{equation}
\begin{aligned}\L_{+,k}v_{k} & \equiv-(1+2u^{2})\left(v_{k}^{\prime\prime}+\frac{N-1}{r}v_{k}^{\prime}-\frac{\la_{k}}{r^{2}}v_{k}\right)-4uu^{\prime}v_{k}^{\prime}+\om v_{k}\\
 & \qquad-(2u\De u+\De u^{2}+pu^{p-1})v_{k}=0.
\end{aligned}
\label{operator: L-+--k}
\end{equation}
For a detailed calculation of $\L_{+,k}$, we refer to Selvitella
\cite{Selvitella-2015}. Note the fact that 
\begin{eqnarray*}
\pa_{x_{l}}u=u^{\prime}(|x|)\frac{x_{l}}{|x|}=u^{\prime}(r)Y_{1,l} &  & \text{for }1\le l\le N.
\end{eqnarray*}
Thus to prove (\ref{eq: kernel 1}), it is sufficient to prove that
\begin{eqnarray}
\L_{+,0}v_{0}=0 &  & \text{if and only if }v_{0}\equiv0,\label{eq: kernel 1.1}
\end{eqnarray}
and that
\begin{eqnarray}
\L_{+,1}v_{1}=0 &  & \text{if and only if }v_{1}\in\text{span}\left\{ u^{\prime}\right\} ,\label{eq: kernel 1.2}
\end{eqnarray}
and that
\begin{eqnarray}
\L_{+,k}v_{k}=0 &  & \text{if and only if }v_{k}\equiv0\label{eq: kernel 1.3}
\end{eqnarray}
for all $k\ge2$.

(\ref{eq: kernel 1.2}) and (\ref{eq: kernel 1.3}) can be proved
in the same way as that of (\ref{eq: kernel of L---0}) and (\ref{eq: kernel of L---k}).
Consider $k=1$. In this case we have $\la_{1}=N-1$. We deduce from
$\L_{+}\pa_{x_{1}}u=0$ that $\L_{+,1}u^{\prime}=0$. Since $u^{\prime}(r)<0$
for all $r>0$, we conclude in a standard way that $u^{\prime}$ is
the first eigenfunction and $0$ is the first simple eigenvalue of
$\L_{+,1}$. This proves (\ref{eq: kernel 1.2}). To conclude (\ref{eq: kernel 1.3}),
it is enough to notice that $\L_{+,k}>\L_{+,1}$ for any $k>1$. This
proves (\ref{eq: kernel 1.3}). 

We leave the proof of (\ref{eq: kernel 1.1}) in the next subsection.
The proof of Proposition \ref{prop: two kernels} is complete. Thus
the proof of Theorem \ref{thm: main result} is complete. \end{proof} 

We remark that (\ref{eq: kernel 2}) can be proved in a more compact
way. Indeed, note that $\L_{-}u=0$ since $u$ solves equation (\ref{eq: Object equ.}).
Thus $u$ is an eigenfunction of $\L_{-}$ with eigenvalue 0. Moreover,
recall that $u$ is a positive eigenfunction. We can conclude in a
standard way that $0$ is the first eigenvalue of $\L_{-}$ and is
simple. Hence $\Ker\L_{-}=\text{span}\{u\}$. See similar discussions
in Chang et al. \cite{Chang et al-2007}.

\subsection{Proof of (\ref{eq: kernel 1.1}) }

Let us first briefly review the proof of (\ref{eq: kernel 1.1}) of
Selvitella \cite{Selvitella-2015}. Suppose that $v_{0}$ belongs
to $L^{2}(\R_{+},r^{N-1}\D r)$, $v_{0}\not\equiv0$ and satisfies
$\L_{+,0}v_{0}=0$. His proof (see Lemma 4.4 of Selvitella \cite{Selvitella-2015})
contains two ingredients. First he proved that $v_{0}(r)$ changes
sign at least once for $r>0$, and then by the disconjugacy interval
argument of Kwong \cite{Kwong1989} he deduced that $v_{0}(r)$ is
unbounded for $r>0$ sufficiently large, which contradicts to $v_{0}\in L^{2}(\R_{+},r^{N-1}\D r)$.
In this way Selvitella \cite{Selvitella-2015} proved (\ref{eq: kernel 1.1}).
To prove that $v_{0}$ changes sign at least once on $\R_{+}$, Selvitella
\cite{Selvitella-2015} used an ordinary differential equation analysis,
in which the assumption $p\ge3$ is needed (see Section 3 of Selvitella
\cite{Selvitella-2015}). While the disconjugacy interval argument
applies to the whole range $1<p<p_{\max}$. 

Taking into account above review, we infer that (\ref{eq: kernel 1.1})
can be deduced from the following result together with the disconjugacy
interval argument as that of Kwong \cite{Kwong1989} and Selvitella
\cite{Selvitella-2015}. 

\begin{proposition} \label{prop:  changing sign proposition} Let
$\L_{+,0}$ be defined as in (\ref{operator: L-+--k}) with $k=0$.
Suppose that $v$ belongs to $L^{2}(\R_{+},r^{N-1}\D r)$, $v\not\equiv0$
and satisfies $\L_{+,0}v=0$. Then $v(r)$ changes sign at least once
for $r>0$. \end{proposition}

Proposition \ref{prop:  changing sign proposition} can be viewed
as a substitute of Proposition 3.10 of Selvitella \cite{Selvitella-2015}.
We use a spectrum analysis to prove Proposition \ref{prop:  changing sign proposition}. 

First we note that $\L_{+,0}$ is the restriction of $\L_{+}$ on
the sector $L_{\rad}^{2}(\R^{N})$, the subspace of radial functions
in $L^{2}(\R^{N})$. Indeed, for any $v\in L_{\rad}^{2}(\R^{N})$,
we have
\[
\begin{aligned}\L_{+}v & =-\De v-2u\De(uv)+\om v-(\De u^{2}+pu^{p-1})v\\
 & =-(1+2u^{2})\left(v^{\prime\prime}+\frac{N-1}{r}v^{\prime}\right)-4uu^{\prime}v^{\prime}+\om v-(2u\De u+\De u^{2}+pu^{p-1})v\\
 & =\L_{+,0}v
\end{aligned}
\]
 since $\la_{0}=0$. Thus we immediately find the following result
which is equivalent to Proposition \ref{prop:  changing sign proposition}. 

\begin{proposition}\label{prop: equiv result} Suppose that $v\in\Ker\L_{+}\cap L_{\rad}^{2}(\R^{N})$
is a nontrivial function. Then $v(x)=v(r)$ with $r=|x|$ changes
sign at least once for $r>0$. \end{proposition}

The idea to prove Proposition \ref{prop: equiv result} is as follows.
Note that $0$ belongs to the spectrum $\si(\L_{+})$ of $\L_{+}$,
since it is straightforward to verify that 
\[
\text{span}\left\{ \pa_{x_{1}}u,\cdots,\pa_{x_{N}}u\right\} \subset\Ker\L_{+}.
\]
We will show that $0$ belongs to the discrete spectrum $\si_{\text{disc}}(\L_{+})$
of $\L_{+}$, that is, $0$ is an isolated eigenvalue of $\L_{+}$
and the corresponding eigenfunction space is finite dimensional. We
also show that $0$ is not the first eigenvalue of $\L_{+}$. Then
we have $\int_{\R^{N}}ve_{1}\D x=0$, where $e_{1}$ is the first
eigenfunction of $\L_{+}$. This fact will imply that $v=v(r)$ changes
sign for $r>0$, once we prove that $e_{1}$ does not change sign
in $\R^{N}$. 

It is easy to verify that $\L_{+}$ is a self-adjoint operator acting
on $L^{2}(\R^{N})$ with form domain $\X_{\C}$ and domain $H^{2}(\R^{N})$.
Hence we have $\si(\L_{+})\subset\R$. Furthermore, by Weyl's theorem
(see Theorem 7.2 of Hislop and Sigal \cite{Hislop-Sigal-1996}) we
have $\si(\L_{+})=\si_{\text{disc}}(\L_{+})\cup\si_{\text{cont}}(\L_{+})$,
and $\si_{\text{disc}}(\L_{+})\cap\si_{\text{cont}}(\L_{+})=\emptyset$,
where $\si_{\text{cont}}(\L_{+})$ denotes the continuous spectrum
of $\L_{+}$. Let us now start the proof of Proposition \ref{prop: equiv result}
with an estimate on $\si_{\text{cont}}(\L_{+})$. Recall that a constant
$\la$ belongs to $\si_{\text{cont}}(\L_{+})$ if and only if there
exists a sequence $\phi_{n}\in H^{2}(\R^{N})$, $n=1,2,\ldots$, such
that 
\begin{eqnarray}
 &  & \|\L_{+}\phi_{n}-\la\phi_{n}\|_{2}\to0\qquad\text{as }n\to\wq,\text{ and }\label{eq: asym. eigenf.}\\
 &  & \|\phi_{n}\|_{2}=1\qquad\text{ for all }n\in\N,\text{ and }\label{eq: normalizations}\\
 &  & \phi_{n}\wto0\qquad\text{ weakly in }L^{2}(\R^{N})\text{ as }n\to\wq.\label{eq: weak convergence}
\end{eqnarray}

\begin{lemma}\label{lem: essential spectrum} We have $\si_{\text{{\rm cont}}}(\L_{+})\subset[\om,\wq).$\end{lemma}
\begin{proof}
Since $\L_{+}$ is self-adjoint, we have $\si(\L_{+})\subset\R$.
So it is sufficient to prove that if $\la<\om$, then $\la\not\in\si_{\text{cont}}(\L_{+})$.
We argue by contradiction. Suppose, on the contrary, that $\la<\om$
is a real number and $\la\in\si_{\text{cont}}(\L_{+})$. Then there
exists a sequence $\{\phi_{n}\}_{n=1}^{\wq}\subset H^{2}(\R^{N})$
such that (\ref{eq: asym. eigenf.})-(\ref{eq: weak convergence})
hold. We claim that, up to a subsequence, 
\begin{eqnarray}
\phi_{n}\to0 &  & \text{strongly in }L^{2}(\R^{N}).\label{esti: estimate 1}
\end{eqnarray}
Then we reach to a contradiction to (\ref{eq: normalizations}) and
Lemma \ref{lem: essential spectrum} is proved. 

We prove (\ref{esti: estimate 1}) as follows. Note that $\De u^{2}+pu^{p-1}$
is bounded in $\R^{N}$ by Theorem \ref{thm: properties of GS}. Thus
we obtain that 
\[
\sup_{n}\int_{\R^{N}}(\om-\la+|\De u^{2}+pu^{p-1}|)|\phi_{n}|^{2}\D x<\wq.
\]
On the other hand, we have
\begin{equation}
\begin{aligned}o(1) & =\langle(\L_{+}-\la)\phi_{n},\phi_{n}\rangle\\
 & =\int_{\R^{N}}\left(|\na\phi_{n}|^{2}+|\na(u\phi_{n})|^{2}+(\om-\la-\De u^{2}-pu^{p-1})|\phi_{n}|^{2}\right)\D x.
\end{aligned}
\label{eq: first order esti}
\end{equation}
The first equality of above follows from (\ref{eq: asym. eigenf.})
and (\ref{eq: normalizations}). Therefore we derive directly from
(\ref{eq: first order esti}) that $|\na\phi_{n}|\in L^{2}(\R^{N})$
is bounded uniformly for all $n\in\N$. Hence $\phi_{n}\in H^{1}(\R^{N})$
is bounded uniformly for all $n$ in view of (\ref{eq: normalizations}).
In particular, we deduce, after possibly passing to a subsequence,
that 
\begin{eqnarray}
\phi_{n}\to0 &  & \text{strongly in }L_{\loc}^{2}(\R^{N}).\label{eq: local strong conver.}
\end{eqnarray}
Next we recall that the function $\De u^{2}+pu^{p-1}$ decays exponentially
to zero at infinity by Theorem \ref{thm: properties of GS}. Combining
this fact together with (\ref{eq: local strong conver.}) gives us
that
\begin{equation}
\int_{\R^{N}}|\De u^{2}+pu^{p-1}||\phi_{n}|^{2}\D x\to0\label{eq: vanishing term 1}
\end{equation}
as $n\to\wq$. Combining (\ref{eq: vanishing term 1}) with (\ref{eq: first order esti})
and recalling that $\om>\la$, we obtain that 
\[
\lim_{n\to\wq}\int_{\R^{N}}|\phi_{n}|^{2}\D x=0,
\]
which contradicts to the assumption (\ref{eq: normalizations}). The
proof of Lemma \ref{lem: essential spectrum} is complete.
\end{proof}
A direct consequence of Lemma \ref{lem: essential spectrum} is that
$0\in\si_{\text{disc}}(\L_{+})$. Lemma \ref{lem: essential spectrum}
also allows us to derive a variational characterization for eigenvalues
of $\L_{+}$ that are below the infimum of $\si_{\text{cont}}(\L_{+})$.
Indeed, suppose that we have eigenvalues 
\[
\inf\si(\L_{+})\equiv\mu_{1}\le\mu_{2}\le\cdots\le\mu_{n}<\inf\si_{\text{{\rm cont}}}(\L_{+}).
\]
The fact $\mu_{1}>-\wq$ follows easily from the elementary estimate
\[
\inf_{\xi\in\X_{\C},\|\xi\|_{2}=1}\langle\L_{+}\xi,\xi\rangle>-\wq.
\]
Then we have 
\[
\mu_{1}=\inf\left\{ \langle\L_{+}\xi,\xi\rangle:\xi\in\X_{\C},\|\xi\|_{2}=1\right\} .
\]
Denoting by $W_{k}$, $2\le k\le n$, the linear space spanned by
the first $n-1$ eigenfunctions corresponding to $\mu_{1}$, $\ldots$,
$\mu_{n-1}$, we have by induction 
\[
\mu_{n}=\inf\left\{ \langle\L_{+}\xi,\xi\rangle:\xi\in\X_{\C},\|\xi\|_{2}=1,\text{ and }\xi\bot W_{n}\right\} .
\]
Furthermore, for any function $\xi\in\X_{\C}$ with $\|\xi\|_{2}=1$,
$\xi\bot W_{n}$ and $\langle\L_{+}\xi,\xi\rangle=\mu_{n}$, $\xi$
is a linear combination of eigenfunctions corresponding to $\mu_{n}$. 

Next we prove that $\L_{+}$ satisfies Perron-Frobenius property.
That is, if $\inf\si(\L_{+})$ is an eigenvalue, then it is simple
and the corresponding eigenfunction can be chosen strictly positive.
In the case of equation (\ref{eq: NLS}), it is easy to verify that
the operator ${\cal A}_{+}$ satisfies the Perron-Frobenius property
as follows. Let $\xi_{1}\in L^{2}(\R^{N})$ is an eigenfunction of
${\cal A}_{+}$ with eigenvalue $a_{1}=\inf\si({\cal A}_{+})<0$.
Then $\xi_{1}$ solves equation 
\[
-\De\xi+(\om-a_{1})\xi_{1}=qQ^{q-1}\xi_{1}.
\]
Since above equation is linear, we can assume that $\xi_{1,+}=\max\{\xi_{1},0\}\not\equiv0$.
Then we obtain, by multiplying above equation by $\xi_{1,+}$, that
$\langle{\cal A}_{+}\xi_{1,+},\xi_{1,+}\rangle=0$, which implies
that $\xi_{1,+}$ is also an eigenfunction of ${\cal A}_{+}$ with
eigenvalue $a_{1}$. Therefore, $\xi_{1,+}$ satisfies the equation
of $\xi_{1}$ as well. Finally, note that $\om-a_{1}>0$. Thus we
have 
\[
\xi_{1,+}=\frac{1}{-\De+(\om-a_{1})}qQ^{q-1}\xi_{1,+}.
\]
Since $Q>0$ for all $x\in\R^{N}$ and since the integral kernel of
the operator $\frac{1}{-\De+(\om-a_{1})}$ is positive away from the
origin, we conclude from above formula that $\xi_{1,+}(x)>0$ for
all $x\in\R^{N}$. Thus $\xi_{1}=\xi_{1,+}$ is strictly positive
in $\R^{N}$. This shows that ${\cal A}_{+}$ satisfies the Perron-Frobenius
property. However, in our case, we do not know whether the first eigenvalue
$\mu_{1}$ of $\L_{+}$ is negative or not. Due to the presence of
the quasilinear term $u\De(u\cdot)$, we can not use above simple
argument to assert that $\L_{+}$ satisfies Perron-Frobenius property.
Nevertheless, we can still deduce the following result. 

\begin{lemma} \label{lem: first eigenvalue is simple and negative}
The first eigenvalue $\mu_{1}$ of $\L_{+}$ is negative and simple.
\end{lemma}

\begin{proof} We have to show that $\mu_{1}<0$ holds and that eigenfunctions
corresponding to $\mu_{1}$ is of constant sign. We argue by contradiction.
Suppose that $\mu_{1}\ge0$ holds. Then the fact $0\in\si_{\text{disc}}(\L_{+})$
implies that $\mu_{1}=0$. Note that $\Ker\L_{+}\neq\emptyset$ is
the eigenfunction space corresponding to $0$. For any $\phi\in\Ker\L_{+}$,
we have that 
\[
-\De\phi-2u\De(u\phi)+\om\phi-(\De u^{2}+pu^{p-1})\phi=0.
\]
Since $u$ is a real valued function, we can assume, with no loss
of generality, that $\phi$ is a real valued function as well. Furthermore,
we can assume that the positive part $\phi_{+}=\max(\phi,0)$ is not
identically zero. Then multiply above equation by $\phi_{+}$. We
obtain by integrating by parts that 
\[
\langle\L_{+}\phi_{+},\phi_{+}\rangle=0.
\]
That is, $\langle\L_{+}\phi_{+},\phi_{+}\rangle$ achieves the first
eigenvalue 0. Thus $\phi_{+}$ is a combination of eigenfunctions
of $0$, which implies that $\phi_{+}$ satisfies equation 
\begin{equation}
-\De\phi_{+}-2u\De(u\phi_{+})+\om\phi_{+}-(\De u^{2}+pu^{p-1})\phi_{+}=0.\label{eq: equ. of first eignef.}
\end{equation}
We claim that equation (\ref{eq: equ. of first eignef.}) implies
that 
\begin{eqnarray}
\phi_{+}(x)>0 &  & \text{for all }x\in\R^{N}.\label{eq: positivity claim}
\end{eqnarray}
Rewrite equation (\ref{eq: equ. of first eignef.}) in the form 
\begin{eqnarray}
-\De\phi_{+}-\sum_{i=1}^{N}b_{i}(x)\cdot\pa_{x_{i}}\phi_{+}+c(x)\phi_{+}=0 &  & \text{in }\R^{N}.\label{eq: equ. of first eignef.2}
\end{eqnarray}
By Theorem \ref{thm: properties of GS}, both functions
\begin{eqnarray*}
b_{i}(x)\equiv-\frac{4u}{1+2u^{2}}\pa_{x_{i}}u,\quad(1\le i\le N) & \text{and} & c(x)\equiv\frac{\om-2u\De u-\De u^{2}-pu^{p-1}}{1+2u^{2}}
\end{eqnarray*}
are bounded smooth functions. Thus elliptic regularity theory gives
us that $\phi_{+}\in C^{\wq}(\R^{N})$ holds. Now, by a famous generalized
comparison principle for second order elliptic equations due to Serrin
(see Theorem 2.10 of Han and Lin \cite[Chapter 2]{Han-Lin-Book}),
we deduce from equation (\ref{eq: equ. of first eignef.2}) that (\ref{eq: positivity claim})
holds. This proves the claim.

Recall that $\pa_{x_{1}}u\in\Ker\L_{+}$. Take $\phi=\pa_{x_{1}}u=u^{\prime}(|x|)x_{1}/|x|$.
Since $u^{\prime}(|x|)<0$ for $|x|>0$, we have that $\phi_{+}(x)\equiv0$
for any $x\in\R^{N}$ with $x_{1}\ge0$. We obtain a contradiction
to (\ref{eq: positivity claim}). Hence we conclude that $\mu_{1}<0$.

Finally, by similar arguments as above, we infer that any eigenfunction
corresponding to $\mu_{1}$ is either positive or negative in $\R^{N}$.
This proves that $\mu_{1}$ is simple. The proof of Lemma \ref{lem: first eigenvalue is simple and negative}
is complete. \end{proof}

Now we are able to prove Propositions \ref{prop:  changing sign proposition}
and \ref{prop: equiv result}.

\begin{proof}[Proof of Propositions \ref{prop:  changing sign proposition} and  \ref{prop: equiv result}]
It is enough to prove Proposition \ref{prop: equiv result} due to
the equivalence. For any function $v\in\Ker\L_{+}\cap L_{\rad}^{2}(\R^{N})$,
$v\not\equiv0$, we obtain from above that 
\[
\int_{\R^{N}}v\bar{e}_{1}\D x=0
\]
holds for any eigenfunction $e_{1}$ of $\L_{+}$ with eigenvalue
$\mu_{1}$. Since $e_{1}$ can be chosen strictly positive in $\R^{N}$,
we infer that $v(x)=v(r)$ with $r=|x|$ must change sign for $r>0$.
This proves Proposition \ref{prop: equiv result}. So follows Proposition
\ref{prop:  changing sign proposition}.\end{proof}

We end this section by showing that $\inf\si(\L_{+})<0$ holds for
$N=2$ via direct computations. Precisely, we show that 
\begin{equation}
\langle\L_{+}u,u\rangle<0.\label{esti: N equals to 2}
\end{equation}
 (\ref{esti: N equals to 2}) follows from a Pohozaev type identity.
Since $u$ solves equation (\ref{eq: Object equ.}), an elementary
calculation gives the following Pohozaev type identity 
\[
\om\int_{\R^{2}}|u|^{2}\D x=\frac{2}{p+1}\int_{\R^{2}}|u|^{p+1}\D x.
\]
Here we used the fact $N=2$. On the other hand, multiplying equation
(\ref{eq: Object equ.}) by $u$ and integrating by parts yields
\[
\int_{\R^{2}}|\na u|^{2}\D x+4\int_{\R^{2}}|u|^{2}|\na u|^{2}\D x+\om\int_{\R^{2}}|u|^{2}\D x=\int_{\R^{2}}|u|^{p+1}\D x.
\]
Recall that 
\[
\langle\L_{+}u,u\rangle=8\int_{\R^{2}}|u|^{2}|\na u|^{2}\D x-(p-1)\int_{\R^{2}}|u|^{p+1}\D x.
\]
Combining above three identities, we deduce that
\[
\langle\L_{+}u,u\rangle=-2\int_{\R^{2}}|\na u|^{2}\D x-\frac{(p-1)^{2}}{p+1}\int_{\R^{2}}|u|^{p+1}\D x<0.
\]
This proves (\ref{esti: N equals to 2}). Thus we conclude that $\inf\si(\L_{+})<0$
holds for $N=2$. 

\emph{Acknowledgment. }
The author is financially supported by the Academy of Finland, project
259224.

\end{document}